\documentclass[a4paper,12pt]{article}
\usepackage{amssymb,amsthm,amsmath,amsfonts}
\usepackage{graphicx}
\usepackage[left=2.5cm,right=2.5cm,top=2cm,bottom=2cm]{geometry}
\usepackage{cite}
\usepackage{color}

 \newtheorem{thm}{Theorem}[section]
 \newtheorem{cor}[thm]{Corollary}
 
 \newtheorem{prop}[thm]{Proposition}
 \theoremstyle{definition}
 \newtheorem{defn}[thm]{Definition}
 \theoremstyle{remark}
 \newtheorem{rem}[thm]{Remark}
 \newtheorem{ex}[thm]{Example}
 \numberwithin{equation}{section}

\begin{document}

\title{A connection between quaternionic slice regular and Fueter hyperholormorphic functions}
\small{
\author{
Jos\'e Oscar Gonz\'alez-Cervantes $^{1}$, Daniel Gonz\'alez-Campos $^{1}$,\\ Juan Bory-Reyes $^{2}$, Baruch Schneider $^{(3){\footnote{corresponding author}}}$}
\vskip 1truecm
\date{\small $^{1}$ Departamento de Matem\'aticas, ESFM-Instituto Polit\'ecnico Nacional. 07338, Ciudad M\'exico, M\'exico\\ Email: jogc200678@gmail.com, daniel$\_$uz13@hotmail.com\\
 {$^{2}$  SEPI-ESIME-IPN-Zacatenco, Instituto Polit\'ecnico Nacional, 07738, Ciudad M\'exico, M\'exico \\ Email: juanboryreyes@yahoo.com\\ $^{(3)}$ Department of Mathematics, University of Ostrava. 70103,
Ostrava, Czech Republic.\\ Email: baruch.schneider@osu.cz}
}

\maketitle

\begin{abstract} 
\small{
This paper presents some results for the slice regular functions and hyperholomorphic functions theories, as consequences of new relationships between the global operator $G$, Fueter operator $\mathcal D$ and Laplacian operator $\Delta_{\mathbb{R}^{4}}$. Although our results are presented in the quaternionic setting, a fine interpretation in terms of vector calculus of the slice regular function theory yields an interesting relationship between this theory and harmonic analysis. Moreover, we introduce and study a submodule  of the hyperholomorphic functions associated to the unit ball $\mathbb B^4(0,1)$ given in terms of new series expansions of these functions.
}
\end{abstract}

\textbf{Math Subject Classification (2020):} 30G35, 46G10.

\section{Introduction}

In 1935, the Swiss mathematician Rudolf Fueter discovered a method for generating regular (holomorphic) functions of a quaternion variable. However, direct powers extension does not guarantee regularity in the hypercomplex sense (monogenicity). Fueter showed that applying Laplace's Delta operator (or its powers extension in higher-dimensional Euclidean spaces) the extended function yields a regular function.

The Fueter theorem is a classical fact from Clifford analysis \cite{BDS,DSS,GSp} going back to the 1930s \cite{F1,F2,Sc,sudbery}. For further works on this topic we refer the reader to \cite{CSSFueter, KQTS,PS, P, PS2, Q, QS2, S}.

Afterwords, Michele Sce proved a generalization of the famous Fueter theorem, called Sce theorem, for functions of a quaternionic variable to non-commutative algebraic structures of higher dimensions, formalizing the concept in quadratic modules and Clifford algebras. Book \cite{CSSSceBook} presents English translations of Michele Sce’s most important works, originally written in Italian during the period 1955-1973, on hypercomplex analysis and algebras of hypercomplex numbers.

The classical function theory for complex-valued functions is based on the so called Cauchy-Riemann operator. In the case of quaternion-valued functions, the theory of null-solutions of the Fueter operator, associated with the standard base of $\mathbb{R}^{4}$, is a widely studied theory, see \cite{F1, F2}. In addition, the Fueter operator factorizes the Laplacian in $\mathbb{R}^{4}$, see \cite{shapiro1,shapiro2,VSMT}. Associated to Fueter operators, quaternionic Stokes, Borel-Pompeiu, and Cauchy integral formulas have been obtained. 

The slice regular function theory is a natural extension of the holomorphic function theory using power series expansions of functions. These functions are often called slice regular when they are quaternion-valued, while when they are Clifford algebra-valued, they are called slice monogenic. A reasonably well overview of historical understanding of how this field developed can be found in \cite{DZ}.

The importance of the theory of quaternionic slice regular functions lies in its direct interpretation through quaternionic series expansions associated to the four-dimensional unit ball.

The slice regular function theory was introduced by Gentili and Struppa in 2006 \cite{GS1,GS2}. This function theory study restricted functions in domains $\Omega\subseteq\mathbb H$ intercepted with copies of the complex plane $\Omega_{{\bf{I}}}=\Omega \cap  \mathbb{C}({\bf{I}})$, for all  ${\bf{I}}$  {unit  vectors in} $ \mathbb{H}$ in which the classical operator of Cauchy Riemann is used. This works attracted the attention of many researchers to a systematic investigation. {The theory of slice regular functions has been connected with the theory of monogenic functions via the Fueter-Sce theorem in \cite{CSSFueter}. Moreover,  In \cite{CAS} the function theories for the fine structures of Dirac type in the quaternionic framework, presenting new series expansions for axially harmonic, Fueter regular, and axially polyanalytic functions are comprehensively investigated}. Much important results can be found, for instance, in \cite{ACS,ACS2, ADS1, ADS2, CSSSceBook, cgs, CSS3, CSS4, GS, GSSReg, gpr}.

In \cite{GP_2, GlobalOp}, was introduced a non-constant coefficients differential operator, whose null solutions are closely related to slice monogenic functions when the domains are given with
suitable conditions. Later on, many researchers started to investigate this global differential operator for slice regularity. In \cite{G, GG1, GG2}, Stokes, Cauchy, and Borel-Pompeiu formulas induced by the global operator were presented. Furthermore, some identities between the global operator $G$ and Moisil$-$Theodorescu and Fueter operators were proved. It should be also mentioned that significant properties on the slice regular function and hyperholomorphic function theories emerge, as consequences of some quaternionic differential equations involving global, Fueter, and Laplacian operators. 
  
The structure of the paper is as follows: Section \ref{prelimi} presents a brief summary of the introductory concepts of  hyperholomorphic function theory and  of slice regular function theory. Subsection \ref{NuyD} shows some   differential equations in which our important quaternionic differential operators are presented. Meanwhile, Subsection \ref{srfunction} and Subsection \ref{hftheor} present our main results for the slice regular function and hyperholomorphic function theories, respectively.

\section{Preliminaries}\label{prelimi} 
The skew field of real quaternions $\mathbb H$ is formed by  $x=x_0+x_{1} {\bf e}_1+x_{2} {\bf e}_2+x_{3} {\bf e}_3$,  where $x_{k}\in \mathbb R$ for  $k= 0,1,2,3$ and the quaternionic  units  satisfy $ {\bf e}_k^{2} =-1$, for   $k=1,2,3$, and $ {\bf e}_1 {\bf e}_2  {\bf e}_3 =-1  $. In addition,  $x_0$   is the scalar part, and ${\bf x}= x_{1} {\bf e}_1+x_{2} {\bf e}_2+x_{3} {\bf e}_3\mapsto (x_{1},x_{2}, x_{3} )$ is the vector part of $x$. 

{Given $x=x_0+{\bf x} , \  y=y_0+{\bf y} \in \mathbb H$ one can see that the quaternionic product allows to see the following identity: 
\begin{align}\label{Hproduct}
xy=& x_0y_0+ x_0 {\bf y}+ y_0{\bf x}- \langle {\bf x}, {\bf y}\rangle_{\mathbb R^3}+ 
 [ {\bf x}, {\bf y} ]_{\mathbb R^3},
\end{align}
where $x_0 {\bf y}$ and $y_0{\bf x}$ are  products by scalar. In addtion, $\langle {\bf x}, {\bf y}\rangle_{\mathbb R^3}$  and $[ {\bf x}, {\bf y} ]_{\mathbb R^3}$ are the scalar and vector product in $ \mathbb R^3 $, respectively. Therefore, the vector product of ${\bf x}, {\bf y}\in \mathbb R^3$ is   
$$[{\bf x}, {\bf y}]_{\mathbb R^3} = \frac{1}{2}({\bf x}{\bf y}- {\bf y}{\bf x}).$$
 
For $x\in \mathbb H$ the mapping of quaternionic conjugation is given by $x\rightarrow {\overline x}:=x_0-x_{1}{\bf e}_1-x_{2}{\bf e}_2-x_{3}{\bf e}_3$. It is easily seen that $x\,{\overline x}={\overline x}\,x=x^{2}_{0}+x^{2}_{1}+x^{2}_{2}+x^{2}_{3}$ and  $\overline { xy}={\overline x}\,\,{\overline y}$ for {$ x, y \in \mathbb H$.} 

The quaternionic scalar product of {$ x, y \in\mathbb H$} is given by 
$$\langle x,y\rangle:=\frac{1}{2}(\bar x y + \bar y x) = \frac{1}{2}(x\bar y + y  \bar x).$$ 

\begin{defn}
The open unit  ball and the unit sphere  in $\mathbb H$ are   denoted by 
$\mathbb B^4(0,1):=\{  x \in \mathbb H \ \mid \  \| x \| < 1\}$ and     
$ \mathbb S^3: = \{     x\in \mathbb  H \ \mid \ \| x\|=1 \}$, respectively.  
 In addition, the unit sphere  in $\mathbb R^3$ is given  by 
  $ \mathbb S^2 := \{    {\bf x} \in \mathbb R^3 \ \mid \ \|  {\bf x} \|=1 \}$. 
\end{defn}

The $\mathbb H$-valued functions $ f$ defined on a  domain $\Omega\subset\mathbb H\cong \mathbb R^4$ are represented by $f=\sum_{k=0}^3 f_k {\bf e}_k$, where $f_k$  is an $\mathbb R$-valued function defined on $\Omega$ for $k=0,1,2,3$. Properties, such as continuity, differentiability, integrability, and so on, that are ascribed to $f$ have to be possessed by all  components $f_k$. We will follow standard notation; for example, $C^{1}(\Omega, \mathbb H)$ denotes the set of continuously differentiable $\mathbb H$-valued functions defined on $\Omega$. 
 
\subsection{A brief summary of hyperholomorphic function theory}\label{hyper}
We recall some basic facts on hyperholomorphic quaternionic-valued functions, see \cite{shapiro1, shapiro2, sudbery}. 
\begin{defn}The left- and the right-Fueter operators are defined by   
$$ \mathcal D[ f] := \sum_{k=0}^3{\bf e}_k \partial_k   f 
 \ \ \textrm{   and   } \ \  \mathcal  D_r[ f  ] :=  \sum_{k=0}^3 \partial_k {  f } {\bf e}_k,\quad 
  \forall f \in C^1(\Omega,\mathbb H),$$ respectively, where $\partial_k   f =\displaystyle \frac{\partial  f}{\partial x_k}$ for all $k$. 
    \end{defn}
The non-commutativity of the quaternionic product implies that  the operators  $\mathcal D$ and $ \mathcal  D_r$ are  generalizations of the Cauchy-Riemann operator. The phenomenon of duality is quite well known in quaternionic analysis, and this is why the function theory induced by $\mathcal D_r$ is not studied  due to its deep relationship and similarity with the hyperholomorphic function theory induced by $\mathcal D$. The quaternionic left-linear module of right-hyperholomorphic function  is defined by $\mathcal M_r (\Omega)= Ker(\mathcal{D}_r)\cap C^1(\Omega, \mathbb H)$.
 
\begin{defn}
Let a domain $\Omega\subset \mathbb H$, then the quaternionic right-linear module of hyperholomorphic function  is defined by 	 $\mathcal M (\Omega)= Ker(\mathcal{D})\cap C^1(\Omega, \mathbb H)$.
Given $f\in C^{1}(\Omega, \mathbb{H})$, and denote $f= f_0 + {\bf f}$. {Then, using the formula 
\eqref{Hproduct} we see that }
 \begin{align*}
 \mathcal{D} [f]=\partial_0 f_0 +\textrm{Grad}[f_0]-\textrm{Div}[{\bf f}]+\textrm{Rot}[{\bf f}],
 \end{align*} 
where
\begin{align*}
\textrm{Grad}[f_0]=& \sum_{k=1}^{3}{\bf e}_{k}\partial_k f_0,\qquad \textrm{Div}[{\bf f}]=\sum_{k=1}^{3} {\bf e}_{k}\partial_k f_k,
\\
\textrm{Rot}[{\bf f}]=&	\bigg(\partial_{2}f_3-\partial_{3}f_2\bigg) {\bf e}_{1}-\bigg(\partial_{1}f_3-\partial_{3} f_1\bigg){\bf e}_{2}+\bigg(\partial_{1} f_2-\partial_{2}f_1\bigg){\bf e}_{3}
\end{align*}
are the well known vector differential operators called gradient, divergence and rotational, respectively.
\end{defn}

\begin{thm} Given a domain  $\Omega\subset \mathbb H$ such that $\partial \Omega$   is a 3-dimensional smooth   surface.   Borel$-$Pompeiu formula shows that 
\begin{align}\label{BorelHyp}  &  \int_{\partial \Omega}(K (y-x)\sigma_{y}   f (y)  +   g(y)   \sigma_{y}  K (y-x) ) \nonumber  \\ 
&  - 
\int_{\Omega} (K  (y-x)  \mathcal D [f] (y) +  \mathcal  D_r [g] (y) K  (y-x)
     )dy   \nonumber \\
		=  &  \left\{ \begin{array}{ll}  f(x) + g(x) , &  x\in \Omega,  \\ 0 , &  x\in \mathbb H\setminus\overline{\Omega},                  
\end{array} \right. 
\end{align} 
and  Stokes' formula shows that 
\begin{align}\label{StokesHyp} \int_{\partial \Omega} g\sigma _x f =  &   \int_{\Omega } \left( g 
  \mathcal  D[f] +  \mathcal  D_r[g] f \right)dx,
\end{align}
for all $f,g \in C^1(\overline{\Omega}, \mathbb H)$, where  $dx$ denotes the differential form of the 4-dimensional volume in $\mathbb R^4$ and  
$$\sigma _{x}:=\left( \sum_{k=0}^3 (-1)^k {\bf e}_k d\hat{x}_k\right)$$ 
is the quaternionic differential form of the 3-dimensional volume in $\mathbb R^4$, where $d\hat{x}_k  = dx_0 \wedge dx_1\wedge dx_2  \wedge  dx_3 $ omitting factor $dx_k$.      {Note that $|\sigma _{x}| = dS_3$  is the differential} form of the 3-dimensional volume in 
$\mathbb R^4$. Let us recall that  
  {$$ K (y- x)=\frac{1}{2\pi^2} \frac{ \overline{y  - x }}{|y  - x |^4}, \quad \forall  y\in \mathbb H\setminus\{ x \}$$}
is called hyperholomorphic Cauchy Kernel. In addition, the integral operator 
$$ \mathcal T[f](x) = \int_{\Omega} K  (y-x) f  (y) dy, \quad \forall f\in L_2(\Omega,\mathbb H)\cup C(\Omega,\mathbb H)$$ satisfies 
\begin{align}\label{FueterInv} \mathcal D \circ \mathcal T[f]=f, \ \ \forall f\in L_2(\Omega,\mathbb H)\cup C(\Omega,\mathbb H). 
\end{align} 
\end{thm}     
     
The following example will be used further. 
\begin{ex}\label{funcionesO}
Let $\Omega\subset \mathbb H$ be  a domain  such that $\partial \Omega$ is a 3-dimensional smooth surface. For $x\in \Omega$ let us introduce the following sequence of functions  
\begin{align*}
 \xi_0 (x) & = 		 { \bf x}  +  3\mathcal T \left[ 1\right]    \\
		 & =   { \bf x} +  3 \int_{\Omega} K  (y-x)  dy, \\
		\xi_n (x) &=	 
		{\bf x}   x^n + \mathcal T\left[     x^n \bigg(  3- 2 {\bf x} x \sum_{k=1}^n |x|^{2(k-1)}  \bigg)  \right]		
		\\
		&=	  { \bf x} x^n      - 2 \sum_{k=1}^{n}\mathcal T \left[   { \bf x}  \ (\bar{x})^{k-1} x^{n-k} \right]+ 3 \mathcal T \left[   x^n   \right] \\
		  &=		  { \bf x} x^n      +
		  \int_{\Omega} K  (y-x) \bigg[ - 2 \sum_{k=1}^{n}
		   \bigg( 
		  {\bf y} \ (\bar{y})^{k-1}  y^{n-k} \bigg) + 3    y^n   \bigg]dy, \quad 
		  \forall x\in \Omega.
\end{align*}  
\end{ex}

\subsection{A brief summary of slice regular function theory}\label{slice}

Here a brief summary of the basic results on slice regular functions will be given,  see \cite{GS1, CSS3, CSS4, GlobalOp}.

\begin{defn}
A domain  $\Omega\subset \mathbb H$ is called axially symmetric if $\Omega\cap \mathbb R\neq \emptyset$  and if  $x+{\bf i}y \in \Omega$ for 
$x,y\in\mathbb R$ and ${\bf i}\in \mathbb S^2$ implies $\{x+{\bf j}y \ \mid  \ {\bf j}\in\mathbb{S}^2\}\subset \Omega$. In addition,   
$\Omega$ is called   slice domain, or s-domain, if $\Omega_{\bf i} = \Omega \cap \mathbb C({\bf i})$ is  a domain in $\mathbb C({\bf i})$, 
 for all ${\bf i}\in\mathbb S^2$.
\end{defn}

The quaternionic  multiplication allows us to see that  ${\bf i}^2=-1$   for all  ${\bf i} \in \mathbb S^2$. Therefore,   $\mathbb C({\bf i}):=\{x+{\bf i}y  \mid  x,y\in\mathbb R\} \cong \mathbb C$ as fields. 

\begin{defn}
Let $\Omega\subset\mathbb H$ be an axially symmetric s-domain open set. A  real differentiable function $f:\Omega\to \mathbb{H}$ is called quaternionic left slice regular function, or slice regular function, if   
\begin{align*}
 \dfrac{\partial  f\mid_{\Omega_{\bf i}}
 }{ \partial \bar z_{\bf i}} 
:=
\frac{1}{2}\left (\frac{\partial}{\partial x}+{\bf i} \frac{\partial}{\partial y}\right )
f\mid_{\Omega_{\bf i}}=0, \  \textrm{ on  $\Omega_{\bf i}= \Omega\cap \mathbb C({\bf i})$,}
\end{align*}
for all ${\bf i}\in \mathbb{S}^2$. Its  slice  derivative, or Cullen's derivative is     
$f'=\displaystyle 
 {\partial}_{{\bf i}}f\mid_{_{\Omega\cap \mathbb C({\bf i})}} = \frac{\partial}{\partial x} f\mid_{_{\Omega\cap \mathbb C({\bf i})}}$. 
 By $\mathcal{SR}(\Omega)$  we mean  the quaternionic right$-$module of   slice regular functions defined on  $\Omega$.
  \end{defn}

The basic examples of slice regular functions are the power series of type $\sum_{n=0}^{\infty} q^n a_n$ on their ball of convergence $\mathbb B^4(0,1)$.

\begin{thm}
Given  $f,g\in\mathcal{SR}(\mathbb B^4(0,1))$ there exist sequences of quaternions  $(a_n)_{n=0}^\infty$ and  $(b_n)_{n=0}^\infty$  such that  
$ 
f(q)= \sum_{n=0}^{\infty} q^n a_n$ and $  g(q)= \sum_{n=0}^{\infty} q^n b_n$  for all $q\in\mathbb B^4(0,1)$.
  \end{thm}

\begin{defn}\label{star-product}
According to the previous theorem, the  $*$-product  of $f$ and $g$ is defined by $(f*g)(q):= \sum_{n=0}^{\infty} q^n  \sum_{k=0}^{n} a_k b_{n-k}$, 
for all $q\in\mathbb B^4(0,1)$. Moreover, if $f(q) \neq, 0$, then $f * g(q) = f(q)g(f(q)^{-1}qf(q))$.
 \end{defn}

\begin{defn}
Let $\Omega\subset \mathbb{H}$ a domain. The so-called global operators $G$  and  $G_r$ were introduced in $\Omega$ as follows:	
\begin{align*}    
G[f] &  := \displaystyle  \|  {\bf  x}   \|^2 \partial_0 f +  {\bf  x}  \sum_{k=1}^3 x_k \partial_k f ,\\
G_r[f]  & :=\|  {\bf x}  \|^2  \partial_0 f +  \sum_{k=1}^3 x_k (\partial_k f ) {\bf x},
\end{align*}
for any continuously differentiable $\mathbb H$-valued function $f$ defined on a domain, where $x$ stands for $\sum_{k=0}^3x_k {\bf e}_k$. 
\end{defn}

We emphasize that $\textrm{Ker}(G)$ is a quaternionic-right module, and $\textrm{Ker}(G_r)$  is a  quaternionic-left module. The quaternionic-right module of slice regular functions  is contained in  $\textrm{Ker}( G)$ on axially symmetric s-domains.

\begin{defn}(Quaternionic holomorphic maps). 
Let $\Omega\subset\mathbb  H$ be an axially symmetric open set and $\mathcal U\subset \mathbb R^{2}$ be such that $q = u + I v \in  \Omega$ whenever $(u, v) \in \mathcal  U$. The (left) quaternionic holomorphic maps on $U$ are of the form $f(q) = \alpha(u, v) +I \beta (u, v)$, where $\alpha, \beta$ are $\mathbb H$-valued differentiable functions, satisfying $\alpha (u, v) = \alpha(u,-v)$ and $\beta(u, v) = - \beta (u,-v)$ for all $(u, v) \in \mathcal U$ (called  even-odd conditions). Also $\alpha$ and $\beta$ satisfy the Cauchy-Riemann system:
 \begin{align*}
 \frac{\partial \alpha}{\partial u}- \frac{\partial \beta}{\partial v} = 0, \quad 
 \frac{\partial \alpha}{\partial v}+ \frac{\partial \beta}{\partial u} = 0.
\end{align*}
\end{defn}
By $\mathcal{HR}(\Omega)$ we mean the set of quaternionic holomorphic maps on $\Omega$, see \cite{GPA}.

For achieving our goals will be necessary to apply the following theorem.
\begin{thm} \label{teoHGM} \ {} 
\begin{enumerate}
\item Let $\Omega\subset \mathbb H$ be an axially symmetric s-domain.
Then $\mathcal {SR}(\Omega)= \mathcal { HR}(\Omega)  \subset  \textrm{Ker}( G) \cap C^1 (\Omega,\mathbb H)$  as function sets.
\item Let $\Omega\subset \mathbb H$ be a  symmetric s-domain that does not intersect
the real axis and $f : \Omega\to  \mathbb H$ be a real differentiable function. Then $\mathcal {SR}(\Omega)=\textrm{Ker}(G) \cap C^1 (\Omega,\mathbb H)$.
\end{enumerate}
\end{thm}
\begin{proof}
See \cite{GlobalOp}.
\end{proof}
In \cite{GCS} one can see the Cauchy Formula  for slice regular functions.

\begin{prop}\label{FueterYnu}  
\begin{align*}  \mathcal   D[{ \bf x } f ] = & 2\frac{{ \bf x }}{\|{ \bf x } \|^2}      G[f] - 3 f - { \bf x }   \mathcal D[f], \\ 
   \mathcal  D_r[ f  { \bf x } ] = & 2     G_r[f] \frac{{ \bf x }}{\|{ \bf x }\|^2}  - 3 f -     { \bf x } D_r[f]  { \bf x } ,\end{align*}
for all $f\in C^1(\Omega,\mathbb H)$
\end{prop}
\begin{proof}
See \cite{G}.
\end{proof}

A study of the function theory associated  with the quaternionic differential operator 
\begin{align*}
	\overline{\vartheta}= \frac{1}{2} \bigg( \partial_0   +\frac{{ \bf x }}{\|{ \bf x }\|^2}    \sum_{k=1}^3 x_k \partial_k\bigg)    
\end{align*}
was presented in \cite{GP_2}.
We define the operator $\nu$ 
\begin{align*}   \nu :=- {\bf x}   \partial_0   +    \sum_{k=1}^3 x_k \partial_k,    
\end{align*}
and note that
\begin{align*}
		\overline{\vartheta}[f](x)=\frac{1}{2}\frac{{ \bf x }}{\|{ \bf x }\|^2}  \nu[f](x),
\end{align*}
\begin{align*}
	G[f](x)=2\|{ \bf x }\|^2\overline{\vartheta}[f](x)
\end{align*}
and
 $$G[f](x) ={\bf x}\nu [f] (x), \quad \forall x\in \Omega, $$ for all $f\in C^1(\Omega,\mathbb H) $. Therefore, 
 \begin{align}\label{gnu}
 Ker(G)\cap C^1(\Omega,\mathbb H) = Ker (\nu) \cap C^1(\Omega,\mathbb H).
\end{align}

\section{Main results}
Our main results show specific relationships between $\nu$, $\mathcal{D}$ and $\Delta_{\mathbb{R}^{4}}$.
\subsection{On $\nu$ and  $\mathcal{D}$}\label{NuyD}
The following statements shows some relationships between $\nu$ and $\mathcal{D}$ to obtain important properties of $\nu$ and  $\textrm{Ker}(\nu)$.
	\begin{prop}\label{identities}Let $\Omega \subset \mathbb H$ be a domain and $f\in C^1(\Omega,\mathbb H)$. Then
		\begin{enumerate}
			\item   $				-2  \nu [f]= { \bf x}\mathcal{D}[f]+\mathcal{D}[ { \bf x} f] + 3 f$,			\item   $			\nu [f]= -{ \bf x}\mathcal{D}[f]+\sum_{k=0}^3[ {\bf x}, \textrm{Grad} f_k  ]{\bf e}_k$.
		\end{enumerate} 
\end{prop}
\begin{proof} Set $f\in C^{1}(\Omega,\mathbb H)$.
	\begin{enumerate}
		\item   {From  Proposition \ref{FueterYnu},  we see that
		\begin{align*} 
			 \|{     \bf x} \|^2\mathcal{D}[f] =2 G[f] +{\bf x} \mathcal{D}[{ \bf x}f]+3{\bf x}f, \quad \textrm{on} \ \  \Omega,
		\end{align*}
		  and  as  $G={\bf{x}\nu}$, hence} 
		\begin{align*} 
			 \|{     \bf x} \|^{2}\mathcal{D}[f] =& 2 {\bf{x}\nu} [f] +{\bf x} \mathcal{D}[{ \bf x}f]+3{\bf x}f,  \quad \textrm{on} \ \  \Omega. 
		\end{align*}
	Then
		\begin{align*} 
			2 \nu[f]= -{ \bf x} \mathcal{D}[f]- \mathcal{D}[{ \bf x}f]-3f,  \quad \textrm{on} \ \  \Omega.
	\end{align*}
	\item   By the other hand,
		\begin{align*}
			\nu[f]&=-{ \bf x}\partial_{0}[f]+\sum_{k=1}^{3}x_{k}\partial_{k} [f]=-{ \bf x}\partial_{0}[f]-\frac{1}{2}\sum_{k=1}^{3}({ \bf x}{\bf{e}}_{k}+{\bf{e}}_{k}{ \bf x})\partial_{k} [f] \\
					 &=-{\bf {x}}\mathcal{D}[f]+\frac{1}{2}\sum_{k=1}^{3}({{ \bf x}{\bf{e}}_{k}-\bf{e}}_{k}{ \bf x})\partial_{k} [f]\\ 
						 &=-{\bf {x}}\mathcal{D}[f]+\frac{1}{2}\sum_{j=0}^{3}\bigg({ \bf x}\bigg(\sum_{k=1}^{3}{\bf{e}}_{k}\partial_{k}[f_{j}]\bigg)-\bigg(\sum_{k=1}^{3}{\bf{e}}_{k}\partial_{k}[f_{j}]\bigg){ \bf x}\bigg){\bf{e}}_{j}\\ 
						 &=-{\bf {x}}\mathcal{D}[f]+\frac{1}{2}\sum_{j=0}^{3}\bigg({ \bf x}\textrm{Grad} f_j-\textrm{Grad} f_j { \bf x}\bigg){\bf{e}}_{j}\\ 
&=-{\bf {x}}\mathcal{D}[f]+\sum_{j=0}^{3}[{ \bf x},\textrm{Grad} f_j]{\bf{e}}_{j}.
		\end{align*}
	\end{enumerate}
\end{proof}

{\begin{cor}\label{vectorInter}Vector interpretation of operator $\nu$. 
 Let $f\in C^1(\Omega,\mathbb H)$, then
	\begin{align*}
		\nu [f] &= \langle {\bf {x}} ,\textrm{Grad}f_{0}+\partial_0 {\bf f}+\textrm{Rot}[{\bf f}]\rangle-(x_2 \partial_3 f_1 -x_3\partial_2 f_1)\\ &+(x_1 \partial_3 f_2 -x_3\partial_1 f_2)-(x_1 \partial_2 f_3 -x_2\partial_1 f_3)\\ &-(\partial_{0} f_{0} +\textrm{Div}[{\bf f}]) {\bf {x}}  -[{\bf {x}},\partial_0 {\bf f}+\textrm{Rot}[{\bf f}]]   \\ & -((x_1 \partial_2 f_2 -x_2\partial_1 f_2)+(x_1 \partial_3 f_3 -x_3\partial_1 f_3)){\bf{e}}_{1}\\ &  +((x_1 \partial_2 f_1 -x_2\partial_1 f_1) -(x_2 \partial_3 f_3 -x_3\partial_2 f_3)){\bf{e}}_{2}\\ &  +((x_1 \partial_3 f_1 -x_3\partial_1 f_1) +(x_2 \partial_3 f_2 -x_3\partial_2 f_2)) {\bf{e}}_{3}.
	\end{align*}
		\end{cor}
\begin{proof}
As  $\mathcal{D}[f]= \partial_0 f_0+ \textrm{Grad} f_0+ \partial_0 {\bf{f}} -  \textrm{Div}[{\bf{f}}]  + \textrm{Rot}[{\bf{f}}]  $ and using the second identity of the previous proposition  
 we have that
 {\begin{align*}
-{\bf {x}}\mathcal{D}[f]+\sum_{j=0}^{3}[{ \bf x},\textrm{Grad} f_j]{\bf{e}}_{j} &= -{\bf {x}} \big(\partial_0 f_0 + \textrm{Grad} f_0 + \partial_0 {\bf f} -  \textrm{Div}[{\bf f}]  + \textrm{Rot}[{\bf f}]\big)  \\ & +\sum_{j=0}^{3}[{ \bf x},\textrm{Grad} f_j]{\bf{e}}_{j}.
\end{align*}}
Then
 {
\begin{align*}
\nu[f] &=-\partial_{0} f_{0}  {\bf {x}} -{\bf {x}}\textrm{Grad} f_0 -{\bf {x}}\partial_0 {\bf f} +  \textrm{Div}[{\bf f}]{\bf {x}}  -{\bf {x}}\textrm{Rot}[{\bf f}]+\sum_{j=0}^{3}[{ \bf x},\textrm{Grad} f_j]{\bf{e}}_{j} \\ & =-\partial_{0} f_{0}  {\bf {x}} +\langle {\bf {x}},\textrm{Grad}f_{0}\rangle -[{\bf {x}},\textrm{Grad} f_{0}] +\langle {\bf {x}},\partial_0 {\bf f}\rangle -[{\bf {x}},\partial_0 {\bf f}]  \\ &+  \textrm{Div}[{\bf f}]{\bf {x}}  +\langle {\bf {x}},\textrm{Rot}[{\bf f}]\rangle -[{\bf {x}},\textrm{Rot}[{\bf f}]]  +[{ \bf x},\textrm{Grad} f_0]{\bf{e}}_{0}+\sum_{j=1}^{3}[{ \bf x},\textrm{Grad} f_j]{\bf{e}}_{j}\\ & =-\partial_{0} f_{0}  {\bf {x}} +\langle {\bf {x}},\textrm{Grad}f_{0}\rangle  +\langle {\bf {x}},\partial_0 {\bf f}\rangle -[{\bf {x}},\partial_0 {\bf f}]  +  \textrm{Div}[{\bf f}]{\bf {x}}  +\langle {\bf {x}},\textrm{Rot}[{\bf f}]\rangle\\ & -[{\bf {x}},\textrm{Rot}[{\bf f}]]  +[{ \bf x},\textrm{Grad} f_1]{\bf{e}}_{1}+[{ \bf x},\textrm{Grad} f_2]{\bf{e}}_{2}+[{ \bf x},\textrm{Grad} f_3]{\bf{e}}_{3}\\ & = \langle {\bf {x}},\textrm{Grad}f_{0}+\partial_0 {\bf f}+\textrm{Rot}[{\bf f}]\rangle-\partial_{0} f_{0}  {\bf {x}}  +  \textrm{Div}[{\bf f}]{\bf {x}} -[{\bf {x}},\partial_0 {\bf f}+\textrm{Rot}[{\bf f}]]  \\ &  +[{ \bf x},\textrm{Grad} f_1]{\bf{e}}_{1}+[{ \bf x},\textrm{Grad} f_2]{\bf{e}}_{2}+[{ \bf x},\textrm{Grad} f_3]{\bf{e}}_{3}\\ & = \langle {\bf {x}},\textrm{Grad}f_{0}+\partial_0 {\bf f}+\textrm{Rot}[{\bf f}]\rangle-(\partial_{0} f_{0}  +  \textrm{Div}[{\bf f}] ) {\bf {x}} -[{\bf {x}},\partial_0 {\bf f}+\textrm{Rot}[{\bf f}]]  \\ &  +((x_2 \partial_3 f_1 -x_3\partial_2 f_1) {\bf{e}}_{1}-(x_1 \partial_3 f_1 -x_3\partial_1 f_1){\bf{e}}_{2}+(x_1 \partial_2 f_1 -x_2\partial_1 f_1){\bf{e}}_{3}){\bf{e}}_{1}\\ &  +((x_2 \partial_3 f_2 -x_3\partial_2 f_2) {\bf{e}}_{1}-(x_1 \partial_3 f_2 -x_3\partial_1 f_2){\bf{e}}_{2}+(x_1 \partial_2 f_2 -x_2\partial_1 f_2){\bf{e}}_{3}){\bf{e}}_{2}\\ &  +((x_2 \partial_3 f_3 -x_3\partial_2 f_3) {\bf{e}}_{1}-(x_1 \partial_3 f_3 -x_3\partial_1 f_3){\bf{e}}_{2}+(x_1 \partial_2 f_3 -x_2\partial_1 f_3){\bf{e}}_{3}){\bf{e}}_{3},
\end{align*}
and
\begin{align*}
	\nu[f] &= \langle {\bf {x}} ,\textrm{Grad}f_{0}+\partial_0 {\bf{f}}+\textrm{Rot}[{\bf {f}}]\rangle-(x_2 \partial_3 f_1 -x_3\partial_2 f_1)\\ &+(x_1 \partial_3 f_2 -x_3\partial_1 f_2)-(x_1 \partial_2 f_3 -x_2\partial_1 f_3)\\ &-(\partial_{0} f_{0} +\textrm{Div}[{\bf f}]) {\bf {x}}   -[{\bf {x}},\partial_0 {\bf {f}}+\textrm{Rot}[{\bf{f}}]]   \\ & -((x_1 \partial_2 f_2 -x_2\partial_1 f_2)+(x_1 \partial_3 f_3 -x_3\partial_1 f_3)){\bf{e}}_{1}\\ &  +((x_1 \partial_2 f_1 -x_2\partial_1 f_1) -(x_2 \partial_3 f_3 -x_3\partial_2 f_3)){\bf{e}}_{2}\\ &  +((x_1 \partial_3 f_1 -x_3\partial_1 f_1) +(x_2 \partial_3 f_2 -x_3\partial_2 f_2)) {\bf{e}}_{3}.
\end{align*}}

\end{proof}}

\begin{rem}
 Using the identity $G={\bf{x}}\nu$,   we can also find a vector   interpretation of the operator $G$. For example:
 {\begin{align*}
	G[f]& =\|x\|^{2}\partial_{0}f_{0}-\langle {\bf{x}}, x_{1}\partial_{1} {\bf{f}}+x_{2}\partial_{2} {\bf{f}}+x_{3}\partial_{3}{\bf{f}} \rangle+\|x\|^{2}\partial_{0}{\bf{f}}\\ &+\big(x_{1} \partial_{1}f_{0}+x_{2} \partial_{2}f_{0}+x_{3} \partial_{3} f_{0}\big){\bf{x}}+\big[{\bf{x}}, x_{1}\partial_{1} {\bf{f}}+x_{2}\partial_{2} {\bf{f}}+x_{3}\partial_{3} {\bf{f}}\big],
\end{align*}}
 for all $f\in C^1(\Omega,\mathbb H)$.
\end{rem}

\begin{prop}\label{inverseThe} Let $\Omega \subset \mathbb H$ be a domain. If ${\bf{x}}\mathcal D[f](x) + 3 f(x)\in L^{2}(\Omega,\mathbb{H})$, then 
$$-2\nu [f]= \mathcal D \bigg[ { \bf x} f + \mathcal T[{ \bf x}\mathcal D[f] + 3 f]\bigg],$$ 
for all $f\in C^1(\Omega,\mathbb H)$.
\end{prop}
\begin{proof}Using the property 
    \eqref{FueterInv} of the integral operator $\mathcal T$  we obtain 
\begin{align*}
	\mathcal{D}\bigg[ { \bf x} f + \mathcal T[{ \bf x}\mathcal D[f]   + 3 f]\bigg]=\mathcal{D}[{ \bf x} f ]+\mathcal{D}\bigg[\mathcal T[{ \bf x}\mathcal{D} [f]   + 3 f]\bigg]=\mathcal{D}[{ \bf x} f ]+{ \bf x}\mathcal{D}[f]   + 3 f. \\
\end{align*}
 and using the first equality in Proposition \ref{identities}  we get the result. 
\end{proof}

Given a domain $\Omega \subset \mathbb{H}$ such that  its boundary  is a 3-dimensional smooth surface. Next,  we will show the participation of the $\nu$ operator in important quaternionic integral formulas in which the function on the variable $x$ obtained by applying the operator $\mathcal T$ to the function $q\mapsto { \bf q}\mathcal D[f](q) + 3 f(q)$ on the variable $q$ is denoted by $\mathcal T\bigg[{\bf q}\mathcal D[f](q) + 3f(q)\bigg](x)$ for $f\in C^1(\Omega,\mathbb H) \cap C(\overline{\Omega},\mathbb H)$.

\begin{prop}\label{formulatipoBorel} 
Let $\Omega \subset \mathbb{H}$ be a domain such that  $\partial \Omega$ is a 3-dimensional smooth surface and $f\in C^1(\Omega,\mathbb H) \cap C(\overline{\Omega},\mathbb H)$. If ${\bf{x}}\mathcal D[f](x)   + 3 f(x)\in L^{2}(\Omega,\mathbb{H})$, then  
\begin{align*}  &  \int_{\partial \Omega}K (y-x)\sigma_{y}   \bigg( { \bf y} f (y) + \mathcal{T}\big[{ \bf{q}}\mathcal D[f](q)   + 3 f(q)\big]  (y) \bigg)     +2 
\int_{\Omega}  K  (y-x)   \nu [f] (y)  dy  \\
& -  \mathcal T\big[{ \bf q}\mathcal D[f] (q)  + 3 f (q)\big] (x) =  { \bf x} f(x) ,  
\end{align*} 
if  $x\in \Omega$. On the other hand, 
\begin{align*}  &  \int_{\partial \Omega} K (y-x)\sigma_{y}   \bigg( { \bf y} f (y) + \mathcal T[{ \bf q}\mathcal D[f](q)   + 3 f(q)]  (y) \bigg)    =- 2 
\int_{\Omega}  K  (y-x)   \nu [f] (y)  dy   ,
\end{align*} 
if  $ x\in \mathbb H\setminus\overline{\Omega}$. In addition,  
\begin{align*} \int_{\partial \Omega}  \sigma _x \bigg( 
{ \bf x} f (x) + \mathcal T\big[{ \bf q}\mathcal D[f](q)   + 3 f(q)\big]  (x)
\bigg)= 
 &  -2 \int_{\Omega } 
  \nu [f] dx.
\end{align*}
\end{prop}

\begin{proof}
 {Suppose $f\in C^1(\Omega,\mathbb H) \cap C(\overline{\Omega},\mathbb H)$. From   Borel$-$Pompeiu formula, see \eqref{BorelHyp} and  the function ${\bf{x}}f(x)+ \mathcal T[{ \bf q}\mathcal D[f] (q)  + 3 f(q)]  (x)$, we have that
\begin{align*}  &  \int_{\partial \Omega} K (y-x)\sigma_{y}   \bigg({\bf{y}}f(y)+ \mathcal T\bigg[{ \bf q}\mathcal D[f] (q)  + 3 f(q)\bigg]\bigg)(y)  \\ 
	&  - 
	\int_{\Omega}  K  (y-x)  \mathcal D \bigg[{\bf{x}}f(x)+ \mathcal T\big[\ { \bf q}\mathcal D[f](q)   + 3 f (q)\ \big] \bigg] (y)   \nonumber \\
	=  &  \left\{ \begin{array}{ll}  {\bf{x}}f(x)+ \mathcal T\big[{ \bf q}\mathcal D[f](q) 
	  + 3 f (q)\big]  (x) , &  x\in \Omega,  \\ 0 , &  x\in \mathbb H\setminus\overline{\Omega}.                     
	\end{array} \right. 
\end{align*}
 Proposition \ref{inverseThe} implies that
	\begin{align*}
			\int_{\partial \Omega}  K  (y-x)  \mathcal D \bigg[{\bf{x}}f(x)+ \mathcal T[{ \bf q}\mathcal D[f] (q)  + 3 f(q)]\bigg] (y)  
		= -2
		\int_{\Omega}K  (y-x)  \nu[f](y). 
	\end{align*}
	In addition,  if $x\in \Omega$
\begin{align*}  &  \int_{\partial \Omega} K (y-x)\sigma_{y} \bigg( {\bf{y}}f(y)+ \mathcal T\bigg[{ \bf q}\mathcal D[f](q)   + 3 f(q)\bigg]  (y) \bigg)  +2
	\int_{\Omega}  K  (y-x)  \nu[f] (y)   \nonumber \\
	=  &   {\bf{x}}f(x)+ \mathcal T\big[{ \bf q}\mathcal D[f](q)   + 3 f(q)\big]  (x).                     
\end{align*}
In the other hand,  {if $ x\in \mathbb H\setminus\overline{\Omega}$, then} 
\begin{align*} \int_{\partial \Omega}  \sigma _x \bigg( 
	{ \bf x} f (x) + \mathcal T\bigg[{ \bf q}\mathcal D[f](q)   + 3 f(q)\bigg]  (x)
	\bigg)+2 \int_{\Omega } 
	\nu [f] dx=0.
	\end{align*}
Therefore, 
\begin{align*}  &  \int_{\partial \Omega} K (y-x)\sigma_{y}   \bigg( { \bf y} f (y) + \mathcal T\bigg[{ \bf q}\mathcal D[f] (q)  + 3 f(q)\bigg]  (y) \bigg)     +2 
	\int_{\Omega}  K  (y-x)   \nu [f] (y)  dy  \\
	& -  \mathcal T\bigg[{ \bf q}\mathcal D[f](q)   + 3 f(q)\bigg] (x) =  { \bf x} f(x) ,  
\end{align*} 
if  $x\in \Omega$, and  
\begin{align*}  &  \int_{\partial \Omega} K (y-x)\sigma_{y}   \bigg(
 { \bf y} f (y) + \mathcal T\bigg[{ \bf q}\mathcal D[f](q)   + 3 f(q)\bigg]  (y) \bigg)    =- 2 
	\int_{\Omega}  K  (y-x)   \nu [f] (y)  dy   ,
\end{align*} 
if  $ x\in \mathbb H\setminus\overline{\Omega}$.
\\
Furthermore, using  Stokes' formula, see \ref{StokesHyp}, with the  functions  1 and ${ \bf x} f (x)+T[{ \bf q}\mathcal D[f](q)   + 3 f(q)](x)$, we have that
\begin{align*} & \int_{\partial \Omega}  \sigma _x \bigg({ \bf x} f (x)+T\bigg[{ \bf q}\mathcal D[f] (q)  + 3 f(q)\bigg](x)\bigg) \\ &=     \int_{\Omega }   
	\mathcal  D\bigg[{ \bf x} f (x)+T\big[{ \bf q}\mathcal D[f](q)   + 3 f(q)\big]\bigg](x)  =-2 \int_{\Omega } 
	\nu [f] dx.
\end{align*}} 
\end{proof}

\begin{cor}\label{formulatipeCauchy}   {Let $\Omega\subset \mathbb H$ be a domain such that $\partial \Omega$ is a 3-dimensional smooth surface.} If $f\in C^1(\Omega,\mathbb H) \cap C(\overline{\Omega},\mathbb H)\cap\textrm{Ker} (\nu)$ and ${\bf{x}}\mathcal D[f](x) + 3 f(x)\in L^{2}(\Omega,\mathbb{H})$,  then  
\begin{align*}  &  \int_{\partial \Omega}K (y-x)\sigma_{y}   \bigg( { \bf y} f (y) + \mathcal T[{ \bf q}\mathcal D[f](q)   + 3 f(q)]  (y) \bigg)     -  \mathcal T[{ \bf q}\mathcal D[f](q)   + 3 f(q)] (x) =  { \bf x} f(x) ,  
\end{align*} 
if  $x\in \Omega$. On the other hand, 
\begin{align*}  &  \int_{\partial \Omega}K (y-x)\sigma_{y}   \bigg( { \bf y} f (y) + \mathcal T[{ \bf q}\mathcal D[f] (q)  + 3 f(q)]  (y)  \bigg)    =- 2 
\int_{\Omega}  K  (y-x)   \nu [f] (y)  dy   ,
\end{align*} 
if  $ x\in \mathbb H\setminus\overline{\Omega}$. Furthermore,   {if   $f\in \textrm{Ker}(\nu)$,  then}
\begin{align*} \int_{\partial \Omega}  \sigma _x \bigg( 
{ \bf x} f (x) + \mathcal T[{ \bf q}\mathcal D[f] (q)  + 3 f(q)]  (x)
\bigg)= 0.
\end{align*}
 \end{cor}
\begin{proof}Follows from Proposition \ref{formulatipoBorel}.
\end{proof}

 \begin{prop}\label{identitiesLaplaceone}  {Let $\Omega\subset \mathbb H$ be a domain. If $f\in C^2(\Omega,\mathbb H)$,  then the following identities hold}
\begin{itemize}
	\item  [1.] $
		-2 \overline{\mathcal D}\bigg[ \nu [f]\bigg]= \overline{\mathcal D}\bigg[{ \bf x}\mathcal D[f]\bigg]+\Delta_{\mathbb R^4}[ { \bf x} f] + 3 \overline{\mathcal D} f$, 
	\item [2.] $
		\Delta_{\mathbb{R}^{4}} [f] =\overline{\mathcal D}\bigg[\dfrac{{ \bf x}}{||{ \bf x}||^{2}}\bigg(\nu [f] -\sum_{k=0}^3[ {\bf x}, \textrm{Grad} f_k  ]{\bf e}_k\bigg)\bigg]$ 
				on $\{x\in  \Omega  \ \mid \ {\bf x}\neq 0\}$.
\end{itemize} 
\end{prop}
\begin{proof} \ \ 
	\begin{itemize}
		\item [1.] The first identity of Proposition \ref{identities} and the decomposition 
		$\overline{\mathcal D}\circ \mathcal{D} = \Delta_{\mathbb R^4}$  {imply} that 
		 \begin{align*}
			-2 \overline{\mathcal D}\bigg[ \nu [f]\bigg]&=\overline{\mathcal D}\bigg[{\bf{x}} \mathcal{D}[f]+\mathcal{D}[{\bf{x}} f]+3f\bigg]\\ &=\overline{\mathcal D}\bigg[{\bf{x}} \mathcal{D}[f]\bigg]+\Delta_{\mathbb R^4}\big[{\bf{x}} f\big]+3\overline{\mathcal D}\big[f\big].
		\end{align*}
		\item [2.]  {From the second identity in Proposition \ref{identities}, we have that}
		\begin{align*}
			\mathcal D[f] =-\frac{{ \bf x}}{||{ \bf x}||^{2}} \bigg(\sum_{k=0}^3[ {\bf x}, \textrm{Grad} f_k  ]{\bf e}_k- \nu[f]\bigg)   
		\end{align*} 
		and applying $\overline{\mathcal D}$ on both sides we see that 
		\begin{align*}
			\Delta_{\mathbb{R}^{4}} [f] =			\overline{\mathcal D}\big[\mathcal D[f]\big] =&-\overline{\mathcal D}\bigg[\frac{{ \bf x}}{||{ \bf x}||^{2}}\bigg(\sum_{k=0}^3[ {\bf x}, \textrm{Grad} f_k  ]{\bf e}_k- \nu[f] \bigg)\bigg].  
		\end{align*}  
	 	\end{itemize}
\end{proof}

\begin{rem}
 {Given $x \in \mathbb H$, from direct computations we can see that} 
\begin{align}\label{sumconj}
\sum_{k=0}^3 e_k x e_k = &  q - x_0- x_1e_1 + x_2e_2 + x_3e_3   
- x_0+ x_1e_1 - x_2e_2 + x_3e_3 - x_0\nonumber \\ 
 & + x_1e_1 + x_2e_2 - x_3e_3 = -2x_0  +2x_1e_1 +2  x_2e_2 +  2  x_3e_3 \nonumber\\
 = & -2\overline{x}   .
\end{align}
\end{rem}
\begin{prop}\label{GDkerConstant}  
 $Ker(G)\cap Ker (\mathcal D) \cap C^1(\mathbb B^4(0,1),\mathbb H)$ consists only of quaternionic constant functions.
\end{prop}
\begin{proof}
  Given   $f\in Ker(G)\cap Ker( \mathcal D) \cap C^1(\mathbb B^4(0,1),\mathbb H)$. Then using   $G={\bf{x}}\nu$  and Proposition 
 \ref{inverseThe}
we obtain  
\begin{align}\label{ecua10101}
0= &  \mathcal D \bigg[ { \bf x} f(x) + \mathcal T[  3 f](x) \bigg] = \mathcal D  [{ \bf x} f  (x)] +  3 f(x) \nonumber\\
 = &   -3 f + \sum_{k=0}^3e_k  { \bf x}(\partial _k f ) (x) +  3 f(x) =   \sum_{k=0}^3 e_k  { \bf x}(\partial _k f ) (x) , 
\quad \forall x\in  \mathbb B^4(0,1).
\end{align}
 As $f\in Ker(G)\cap C^1( \mathbb B^4(0,1),\mathbb H)$, then there exists  a sequence of quaternions $(a_n)_{n\geq 0}$ such that 
$$  f(x) = \sum_{n=0}^{\infty } x^n a_n, \quad \forall x\in \mathbb B^4(0,1). $$ 
Then \eqref{ecua10101} and the uniform convergence of  { $\big( \sum_{n=0}^{m} x^n a_n \big)_{m\geq 0}$} to $f$, inherited from   complex analysis on slices,  {give us} that 
$$ \lim_{m\to \infty }\sum_{n=0}^{m }\sum_{k=0}^3 e_k  {\bf x} \partial _k ( x^n)  a_n =0 ,\quad \forall x\in   \mathbb B^4(0,1),$$
or equivalently
$$ \lim_{m\to \infty }\sum_{n=0}^{m }\sum_{k=0}^3 e_k  {\bf x} (e_k x^{n-1} + x e_k x^{n-2}  + \cdots + x^{n-1} e_k  )a_n =0 ,$$  
 for all $x\in   \mathbb B^4(0,1)$. From \eqref{sumconj}, we obtain that  
 \begin{align*}
& \sum_{k=0}^3 e_k  {\bf x} (e_k x^{n-1} + x e_k x^{n-2}  + \cdots + x^{n-1} e_k  ) \\
=& \sum_{k=0}^3 e_k  {\bf x}  e_k x^{n-1} + \sum_{k=0}^3 e_k  {\bf x} x e_k x^{n-2}  + \cdots + \sum_{k=0}^3 e_k  {\bf x} x^{n-1} e_k  \\
=& -2 \overline{ {\bf x}  }  x^{n-1} -2 \overline{ {\bf x} x}   x^{n-2}  - \cdots -2 \overline{    {\bf x} x^{n-1} }  \\
=&  2\left[   {\bf x}     x^{n-1} +\overline{ x} {\bf x}   x^{n-2}  + \cdots + \overline{   x}^{n-1}   {\bf x} \right]. 
 \end{align*}
 {Then
\begin{align*}0= & \lim_{m\to \infty }\bigg(\sum_{n=0}^{m } 2\left[   {\bf x}     x^{n-1} +\overline{ x} {\bf x}   x^{n-2}  + \cdots + \overline{   x}^{n-1}   {\bf x} \right] a_n\bigg) \\
= &  2{\bf x} \lim_{m\to \infty }\bigg(\sum_{n=0}^{m }  \left[        x^{n-1} +\overline{ x}     x^{n-2}  + \cdots + \overline{   x}^{n-1}     \right] a_{n}\bigg),
\end{align*}
for all $x\in   \mathbb B^4(0,1)$. } 

 {Therefore, 
\begin{align*}0=    &  \lim_{m\to \infty }\bigg(\sum_{n=0}^{m }  \left[        x^{n-1} +\overline{ x}     x^{n-2}  + \cdots + \overline{   x}^{n-1}     \right] a_n\bigg),
\end{align*}
for all $x\in   \mathbb B^4(0,1)$.
In particular, if 
  $x= \lambda\in (-1,1)$, then  
 \begin{align*}0=    &  \lim_{m\to \infty }\bigg(\sum_{n=0}^{m }  n\lambda ^{n-1}   a_n \bigg)=
  \lim_{m\to \infty }\bigg(\sum_{n=0}^{m }   \lambda ^{n}   a_n \bigg)' = f'(\lambda).
 \end{align*}
 As $f'$ is the zero function on $(-1,1)$, then  the principle of identity of slice  regular functions, we conclude that $f' $ is the zero function on $\mathbb B^4(0,1)$. Therefore, $f$ is a constant function.}
\end{proof}

\subsection{On the slice regular function theory}\label{srfunction}

We will study the behavior of Fueter and Laplace operators on slice regular functions.

\begin{cor}\label{SRGYD} \ {}
\begin{enumerate}
\item    {If $\Omega\subset \mathbb H$ is an axially symmetric s-domain and    $f\in \mathcal {SR}(\Omega)$,  then}  
\begin{align*} 
{ \bf x}\mathcal D[f]+\mathcal D[ { \bf x} f] + 3 f =  0,\\
 { \bf x}\mathcal D[f]=   \sum_{k=0}^3[{\bf x},\textrm{Grad} f_k  ]{\bf e}_k,
 \end{align*}  on $\Omega$.
\item   Let $\Omega\subset \mathbb H$ 
 be an  symmetric s-domain which does not intersect the real axis,  then the following facts are equivalents:
 \begin{enumerate}   
\item  $f\in \mathcal {SR}(\Omega)$. 
\item  ${\bf x}\mathcal D[f]+\mathcal D[ { \bf x} f] + 3 f =0 $ on $\Omega$.
\item  ${\bf x}\mathcal D[f]=\sum_{k=0}^3[{\bf x},\textrm{Grad} f_k  ]{\bf e}_k$  on $\Omega$.
\end{enumerate}  
\end{enumerate}  
\end{cor}
\begin{proof} 
Direct consequences of Theorem \ref{teoHGM} and Proposition \ref{identities}. 
\end{proof}

\begin{rem} 
Let $\Omega\subset \mathbb H$ be an axially symmetric s-domain and $f\in \mathcal {SR}(\Omega)$, then another vector version of fact 1 is:
 \begin{align*}
 -3f_0 =& -\partial_0\langle {\bf x}, {\bf f}\rangle_{\mathbb R^3}- \textrm{Div}( f_0{\bf x} + [{\bf x}, {\bf f}]_{\mathbb R^3} ) - \langle {\bf x}, \partial_0 {\bf f} + \textrm{Grad}[f_0]+ \textrm{Rot}[{\bf f}]\rangle_{\mathbb R^3}, \\
-3{\bf f} =& \partial_0 ( 2 f_0 {\bf x}  + [{\bf x}, {\bf f}]_{\mathbb R^3}  ) - 
\textrm{Grad}(\langle {\bf x}, {\bf f}\rangle_{\mathbb R^3} )+ \textrm{Rot}( f_0 {\bf x}+ [{\bf x}, {\bf f}]_{\mathbb R^3}  )   + \textrm{Div} [\bf f]  \ {\bf x} - 
[  {\bf x} , \partial_0 {\bf f } + \textrm{Rot} [\bf f] ]_{\mathbb R^3}.
\end{align*}

On the other hand, given $f\in \mathcal {SR}(\mathbb B^4(0,1))$ a non constant function,  suppose that ${\bf x}$  and $\textrm{Grad} f_k (x)$ are two collinear vectors for all $x\in \mathbb B^4(0,1)$ and all  $k\in \{0,1,2,3\} $, or equivalently,
$$[{\bf x},\textrm{Grad} f_k (x) ]= 0 ,\quad \forall x\in \mathbb B^4(0,1),$$
for all $k\in \{0,1,2,3\} $.  Then Proposition \ref{GDkerConstant} implies that ${\bf x}\mathcal D[f] = 0$  on $\mathbb B^4(0,1)$. Therefore, $\mathcal D[f] = 0$  
on $\mathbb B^4(0,1)$. But Proposition \ref{GDkerConstant}  shows that $f$ must be a constant function, which is false. So there is an open set   $U\subset \mathbb B^4(0,1)$ and 
$k\in \{0,1,2,3\} $ such that ${\bf x}$ and $\textrm{Grad} f_k (x)$ are  {not} collinear vectors  for  all $x\in U$. 
\end{rem}

\begin{rem} A vector interpretation  of the slice regular functions. 
 If $\Omega\subset \mathbb H$ is an axially symmetric s-domain, then Corollary \ref{vectorInter} allows us to see that $f\in \mathcal {SR}(\Omega)$ if  
 \begin{align*}
	&  \langle {\bf{x}} ,\textrm{Grad}f_{0}+\partial_0 {\bf{f}}+\textrm{Rot}[{\bf {f}}]\rangle-\big(x_2 \partial_3 f_1 -x_3\partial_2 f_1\big)\\ &+\big(x_1 \partial_3 f_2 -x_3\partial_1 f_2\big)-\big(x_1 \partial_2 f_3 -x_2\partial_1 f_3\big)=0,
\end{align*}
and
\begin{align*}
	&  -\big(\partial_{0} f_{0} +\textrm{Div}[{\bf{f}}]\big) {\bf{x}}   -\big[{\bf{x}},\partial_0 {\bf{f}}+\textrm{Rot}[{\bf {f}}]\big]   \\ & -((x_1 \partial_2 f_2 -x_2\partial_1 f_2)+(x_1 \partial_3 f_3 -x_3\partial_1 f_3)){\bf{e}}_{1}\\ &  +\big((x_1 \partial_2 f_1 -x_2\partial_1 f_1) -(x_2 \partial_3 f_3 -x_3\partial_2 f_3)\big){\bf{e}}_{2}\\ &  +\big((x_1 \partial_3 f_1 -x_3\partial_1 f_1) +(x_2 \partial_3 f_2 -x_3\partial_2 f_2)\big) {\bf{e}}_{3}=0.
\end{align*}
The previous  terms 	are the scalar and vector  parts of  $\nu[f]$, respectively.
\end{rem}
 
The following result shows a hyperholomorphic mapping defined on important quaternionic submodules of $\ker(G)$.

\begin{prop}\label{SRANDM} 
Let $\Omega\subset \mathbb H$ be an axially symmetric s-domain and $f\in \mathcal {SR}(\Omega)$. If
$${\bf x}\mathcal D[f]   + 3 f\in L^{2}(\Omega,\mathbb{H}),$$
then we have ${\bf x} f + \mathcal T\bigg[{ \bf x}\mathcal D[f] + 3 f\bigg]\in \mathcal M(\Omega)$.
  
\noindent
On the other hand, suppose that $\Omega\subset \mathbb H$ is an  symmetric s-domain which does not intersect the real axis. Let $f\in C^{1}(\Omega,\mathbb{H})$  such that ${ \bf x}\mathcal D[f] + 3 f\in L^{2}(\Omega,\mathbb{H})$. Then $f\in \mathcal {SR}(\Omega)$ if and only if  
$${\bf x} f + \mathcal T\bigg[{ \bf x}\mathcal D[f] + 3 f\bigg]\in \mathcal M(\Omega).$$ 
  \end{prop}
\begin{proof} {Use Theorem \ref{teoHGM},  Proposition \ref{inverseThe} and  note that $\nu[f]=0$ in $\Omega$, if} $\mathcal D \bigg[ {\bf x} f + \mathcal T\big[{ \bf x}\mathcal D[f] + 3 f\big] \bigg]=0$ in $\Omega$, i.e., $ { \bf x} f + \mathcal T[{ \bf x}\mathcal D[f]   + 3 f]\in \mathcal M(\Omega)$. 
\end{proof}

The following corollary  shows some properties of the Laplacian on $\mathcal {SR}(\Omega)$. 

\begin{cor} \label{identitiesLaplacetwo} If $\Omega\subset \mathbb H$ is an axially symmetric s$-$domain and $f\in \mathcal {SR}(\Omega)$, then 
   \begin{enumerate}
   	\item   $
   		\Delta_{\mathbb R^4}[ { \bf x} f] =  - \overline{\mathcal D}\bigg[{ \bf x}\mathcal D[f]\bigg]- 3 \overline{\mathcal D} f$ on $\Omega$. 
   	\item   $
   		\Delta_{\mathbb R^4}[ f]  =   \overline{\mathcal D}\bigg[\dfrac{{ \bf x}}{\|{ \bf x}\|^2} 
   		\sum_{k=0}^3[ \textrm{Grad} f_k,{\bf x} ]{\bf e}_k    \bigg]$ on $\{x\in  \Omega  \ \mid \ {\bf x}\neq 0\}$.
   	\item   $
   		\Delta_{\mathbb R^4}\big[\mathcal D [f]\big]  =   \Delta_{\mathbb{R}^{4}}\bigg[\dfrac{{ \bf x}}{||{ \bf x}||^{2}}\bigg( \sum_{k=0}^3[ \textrm{Grad} f_k , {\bf x} ]{\bf e}_k\bigg)\bigg]$   on $\{x\in  \Omega  \ \mid \ {\bf x}\neq 0\}$.
   \end{enumerate}
  \end{cor}
\begin{proof} Recall that   $f\in \mathcal {SR}(\Omega)$ implies $\nu[f]=0$  on $\Omega$. 
 {\begin{enumerate}
		\item   The first equation in Corollary \ref{identitiesLaplaceone} gives us that   
		\begin{align*}
			0=&\overline{\mathcal D}[{ \bf x}\mathcal D[f]]+\Delta_{\mathbb R^4}[ { \bf x} f] + 3 \overline{\mathcal D} f.\end{align*}   
	 		\item  Using the second equation of Corollary \ref{identitiesLaplaceone} we obtain \begin{align*}
			\Delta_{\mathbb{R}^{4}} [f] =-\overline{\mathcal D}\bigg[\frac{{ \bf x}}{||{ \bf x}||^{2}}\bigg( \sum_{k=0}^3[ {\bf x}, \textrm{Grad} f_k  ]{\bf e}_k\bigg)\bigg].  
		\end{align*} 
	\item   Applying  Fueter operator at the previous identity we obtain  
\begin{align*}
{\mathcal D}[\Delta_{\mathbb{R}^{4}} [f] ]={\mathcal D}\bigg[\overline{\mathcal D}\bigg[\frac{{ \bf x}}{||{ \bf x}||^{2}}\bigg( \sum_{k=0}^3[ \textrm{Grad} f_k, {\bf x}  ]{\bf e}_k\bigg)\bigg]\bigg]. 
	\end{align*}
	\end{enumerate}}
\end{proof}

 {The following result  shows an important relationship between the harmonic function theory and the slice regular function theory.} 

\begin{prop} \label{identitiesLaplacetwo2} Let $\Omega\subset \mathbb H$ be   an axially symmetric s-domain and $f\in \mathcal {SR}(\Omega)$. If   
$$\dfrac{{ \bf x}}{||{ \bf x}||^{2}}  \sum_{k=0}^3[ \textrm{Grad} f_k , {\bf x}]{\bf e}_k\in L^{2}(\Omega,\mathbb{H}),$$ then we  
$$f - \mathcal T\bigg[    \dfrac{{ \bf x}}{||{ \bf x}||^{2}}  \sum_{k=0}^3[ \textrm{Grad} f_k , {\bf x} ]{\bf e}_k\bigg]$$
is a $\mathbb H$-valued harmonic function on $\{x\in  \Omega  \ \mid \ {\bf x}\neq 0\}$, i.e., its real components are harmonic functions on $\{x\in  \Omega  \ \mid \ {\bf x}\neq 0\}$.
  \end{prop}
  \begin{proof}Apply the identity  \eqref{FueterInv} on the right side of  the second fact  of Corollary \ref{identitiesLaplacetwo} to obtain that  
  \begin{align*}   		\Delta_{\mathbb R^4}[ f]  = &  \overline{\mathcal D}\circ {\mathcal D}\bigg[ {\mathcal T}\big[\dfrac{{ \bf x}}{\|{ \bf x}\|^2} 
   		\sum_{k=0}^3[ \textrm{Grad} f_k,{\bf x} ]{\bf e}_k    \big]\bigg]\\
   		=&  		\Delta_{\mathbb R^4} \bigg[ {\mathcal T}\big[\dfrac{{ \bf x}}{\|{ \bf x}\|^2} 
   		\sum_{k=0}^3[ \textrm{Grad} f_k,{\bf x} ]{\bf e}_k    \big]\bigg],
   		\end{align*}  on $\{x\in  \Omega  \ \mid \ {\bf x}\neq 0\}$.
    \end{proof}
 
We shall see  some  global integral equations of  slice regular functions. 

\begin{cor} If $\Omega\subset \mathbb H$ is an axially symmetric s$-$domain and $f\in \mathcal {SR}(\Omega)$. If ${ \bf x}\mathcal D[f]   + 3 f\in L^{2}(\Omega,\mathbb{H})$, then  
\begin{align*}  &  \int_{\partial \Omega}K (y-x)\sigma_{y}   
 \bigg({ \bf y} f + \mathcal T\bigg[{ \bf q}\mathcal D[f](q)   + 3 f(q)\bigg] (y) \bigg) -
  \mathcal T \big[\ { \bf q}\mathcal D[f](q)   + 3 f(q) \ \big](x) 
  \\
  =  &   { \bf x} f (x) ,
\end{align*} 
 for all $x\in \Omega$, and                
\begin{align*}  &  \int_{\partial \Omega}K (y-x)\sigma_{y}   
 \bigg({ \bf y} f + \mathcal T\big[{ \bf q}\mathcal D[f](q)   + 3 f(q)\big] (y) \bigg) =0, \quad \forall 
  x\in \mathbb H\setminus\overline{\Omega}.
 \end{align*} 
  \end{cor}
\begin{proof}
  Corollary \ref{formulatipeCauchy} and  Proposition \ref{SRANDM} show that 
$${\bf{x}} f + \mathcal{T}\bigg[{ \bf{q}}\mathcal{D}[f] (q)  + 3 f (q)\bigg]\in \mathcal{M}(\Omega),$$  
 {\begin{align*}   
&	\int_{\partial\Omega}{K(y-x)\sigma_{y}   
	\bigg({\bf{y}}f+\mathcal{T}\big[{\bf{q}}\mathcal{D}[f](q)+3f(q)\big](y)\bigg)-\mathcal{T}\big[{\bf{q}}\mathcal{D}[f](q)+3f(q)\big](x)}\\
	= & {\bf{x}}f(x),
\end{align*} 
for all $x\in \Omega$, and}
\begin{align*}    
	\int_{\partial \Omega}K (y-x)\sigma_{y}   
	\bigg({ \bf{y}}f + \mathcal{T}\big[{ \bf{q}}\mathcal{D}[f](q)+3f(q)\big](y)\bigg) =&0,
	\quad \forall x\in \mathbb{H}\setminus\overline{\Omega}.
\end{align*} 
\end{proof}

\begin{rem}
Recall that, if $\Omega\subset \mathbb H$ is a symmetric s-domain that does not intersect the real axis. Let $f\in C^{1}(\Omega,\mathbb{H})$ and ${ \bf x}\mathcal D[f]   + 3 f\in L^{2}(\Omega,\mathbb{H})$. Then $f\in \mathcal {SR}(\Omega)$ if and  only if  
$${\bf x} f + \mathcal T\bigg[{ \bf x}\mathcal D[f] + 3 f\bigg]\in \mathcal M(\Omega).$$ 
As in the hyperholomorphic function theory  Cauchy-type formula and  Cauchy-type theorem are innate or distinctive characteristics, then we obtain that $f\in \mathcal {SR}(\Omega)$ if and only if
\begin{align*}  &  \int_{\partial \Omega}K (y-x)\sigma_{y}   
 \bigg({ \bf y} f + \mathcal T\bigg[{ \bf q}\mathcal D[f] (q)  + 3 f(q)\bigg] (y) \bigg) -
  \mathcal T \big[\ { \bf q}\mathcal D[f] (q)  + 3 f(q) \ \big](x) \\
  = &    { \bf x} f (x) ,
\end{align*} 
 for all $x\in \Omega$, and                
\begin{align*}  &  \int_{\partial \Omega}K (y-x)\sigma_{y}   
 \bigg({ \bf y} f + \mathcal T\big[{ \bf q}\mathcal D[f] (q)  + 3 f(q)\big] (y) \bigg) =0, \quad \forall 
  x\in \mathbb H\setminus\overline{\Omega},
 \end{align*} 
 respectively.
\end{rem}

\subsection{On the hyperholomorphic function theory}\label{hftheor}

This subsection  presents a study of the  hyperholomorphic function theory as  consequence of the statements presented in the previous two sections.

The Fueter operator applied to a power of the identity function gives:
 \begin{align*}
\mathcal D[x^n]= & \sum_{k=0}^3 e_k\partial_k (x^n)=  \sum_{k=0}^3 e_k ( e_k x^{n-1} + x e_k x^{n-2}+ x^2 e_k    x^{n-3}+ \cdots + x^{n-1} e_k) \\    
= & -2\bigg[   x^{n-1} + \overline{x}  x^{n-2}+ (\bar x)^2 x^{n-3}+ \cdots +  (\bar x)^{n-1} \bigg] = -2\sum_{k=1}^{n} (\bar{x})^{k-1}x^{n-k},  
\end{align*}
for $n\in \mathbb N$, where we have used \eqref{sumconj} in the computation shown in the second row, see \cite{B}. 

\begin{prop}\label{serieg} 
Let $f\in \mathcal {SR}(\mathbb B^4(0,1))$ given by $f(x)= \sum_{n=0}^\infty  x^n a_n$   on $ \mathbb B^4(0,1)$ such that  
 $$\sum_{n=1}^\infty \left\{   \mathcal T \left[  -2  { \bf x}      \sum_{k=1}^{n} (\bar{x})^{k-1} x^{n-k} + 3 x^n \ \right]\right\} a_n $$
is uniformly convergent on compacts subsets of $\mathbb B^4(0,1)$,  then  the function series $\sum_{n
=0}^\infty \xi_n(x)  a_n $   belongs to $\mathcal M ( \mathbb B^4(0,1))$, 
where  the functions $\mathcal \xi_n$ are given by Example \ref{funcionesO}.
\end{prop}
\begin{proof}
Let $f(x)=\sum_{k=0}^{\infty}x^n a_n$ for all $x\in \mathcal {SR}(\mathbb B^4(0,1))$, with $a_n \in \mathbb{H}$. The hypothesis and the identity $\mathcal D\big[x^m \big] =-2\sum_{k=1}^{m}(\bar{x})^{k-1}  x^{m-k}$, for $m\in \mathbb N$, are used to see that 
$$\left(\displaystyle \mathcal T\bigg[{\bf x}\mathcal D\bigg[\sum_{k=0}^{n} x^k a_k \bigg] + 3\sum_{k=0}^{n}x^k a_k  \bigg]\right)_{n\geq 0}$$      
is uniformly convergent on compacts subsets of $\mathbb B^4(0,1)$  to a function denoted by $p$. 

By \eqref{FueterInv}, we obtain that   
$$\left(\displaystyle {\bf x}\mathcal D\bigg[\sum_{k=0}^{n} x^k a_k \bigg] + 3\sum_{k=0}^{n}x^k a_k \right)_{n\geq 0},$$
converges to $\mathcal D p$. Then
${\bf x}\mathcal D[f]+ 3f = \mathcal D p$ and $\mathcal D\bigg[\mathcal T\big[ {\bf x}\mathcal D[f]+ 3f \big] \bigg]= \mathcal D p$, which shows that there exists a hyperholomorphic function $h$ on 
$\mathbb B^4(0,1)$ such that $p=  \mathcal T\big[ {\bf x}\mathcal D[f]+ 3f \big] + h$.
 
Therefore,  
$$\left\{ { \bf x}  +  3\mathcal T \left[ 1\right] \right\}  a_0  + \sum_{n=1}^\infty \left\{  { \bf x} x^n - 2 \sum_{k=1}^{n}\mathcal T \left[{\bf x}\ (\bar{x})^{k-1}  x^{n-k}\right] + 3 \mathcal T \left[ x^n   \right]\right\} a_n$$
 converges to 
 ${\bf x} f(x) + p(x)= { \bf x} f(x)  + \mathcal T\big[ {\bf q}\mathcal D[f](q)+ 3f (q)\big] (x)+ h (x)$. Recall that  Proposition \ref{SRANDM}   shows   that 
$ {\bf x} f + \mathcal T\big[{ \bf q}\mathcal D[f] (q)  + 3 f (q)\big]$ is a hyperholomorphic function too.
\end{proof}

\begin{rem}\label{remsubfam} The previous corollary presents a subfamily of functions  $\{\xi_n \ \mid n \geq 0\}$ of $\mathcal{M}\bigg(\mathbb{B}^{4}(0,1)\bigg)$ which  {allows us} to  describe a new quaternionic submodule of $\mathcal{M}\bigg(\mathbb{B}^{4}(0,1)\bigg)$ formed  by function  series  induced by regular slice functions.  This submodule complements the well-known development of all elements of $\mathcal{M}\bigg(\mathbb{B}^{4}(0,1)\bigg)$, given in terms of families of symmetric polynomials in Fueter's variables, to be found as the solution to a  Gleasson problem, see \cite{A1}. 

On the other hand, if  $f, g\in \mathcal {SR}(\mathbb B^4(0,1))$  given by 
$f(x)= \sum_{n=0}^\infty  x^n a_n$   and  $g(x)= \sum_{n=0}^\infty  x^n b_n$ 
 satisfy Proposition \ref{serieg},  then  the function series $\sum_{n
=0}^\infty \xi_n(x)  a_n $  and 
$\sum_{n
=0}^\infty \xi_n(x)  b_n $, for all $x\in \mathbb B^4(0,1)$,   belong  to $\mathcal M ( \mathbb B^4(0,1))$. In addition, the star-product  $(f\ast g)$ given in Definition \ref{star-product}, as follows: 
$$\displaystyle f\ast g (x)= \sum_{n=0}^\infty  x^n (\sum_{k=0}^n a_k b_{n-k}), \quad \forall x\in \mathbb B^4(0,1), $$
satisfies the hypothesis of Proposition \ref{serieg} and generates the hiperholomorphic function: 
$$ \sum_{n=0}^\infty  \xi_n(x) (\sum_{k=0}^n a_k b_{n-k}), \quad \forall x\in \mathbb B^4(0,1),$$
i.e., the  star-product  of slice regular functions induces   a $\odot$-product  on this quaternionic submodule:
 $$
\bigg(  \sum_{n
=0}^\infty \xi_n(x)  a_n \bigg) \odot \bigg( \sum_{n
=0}^\infty \xi_n(x)  b_n  \bigg) =  \sum_{n=0}^\infty  \xi_n(x) (\sum_{k=0}^n a_k b_{n-k}), \quad \forall x\in \mathbb B^4(0,1).$$
\end{rem}

\begin{rem}
{In \cite{AA, AKA1, AKA2} one can find a connection between the Fueter–Sce theorem and the Cauchy–Kovalevskaya extension}.

In our case, we present, in Proposition \ref{serieg}  that using the theory of regular functions and Theodorescu operator  one can  find a family of hyperholomorphic functions with series expansion of functions given directly by elementary functions in which a product can be defined.
\end{rem}

\begin{ex} \ {}
\begin{enumerate}
\item The functions  $$\displaystyle \text{exp}(x) =\sum_{n=0}^{\infty}\frac{x^{n}}{n!} , \ \ 
 \displaystyle \text{sin}(x) =\sum_{n=0}^{\infty}(-1)^n\frac{x^{2n+1}}{(2n+1)!}  , \ \ 
 \displaystyle \text{cos}(x) =\sum_{n=0}^{\infty}(-1)^n\frac{x^{2n}}{(2n)!}, $$
 for all $x\in \mathbb H$, allow us to consider their hyperholomorphic  versions: 
   $$\displaystyle  \sum_{n
=0}^\infty \xi_n(x)  \frac{1}{n!} , \ \  \displaystyle  \sum_{n
=0}^\infty (-1)^n \xi_{2n+1}(x)  \frac{1}{(2n+1)!} , \ \  \displaystyle  \sum_{n
=0}^\infty (-1)^n \xi_{2n}(x)  \frac{1}{(2n)!} ,$$
 for all $x\in \mathbb B^4(0,1)$,  	respectively.
\item If $f$ is a quaternionic polynomial restricted   to $\mathbb B^4(0,1)$ given by $f(x)=\sum_{n=0}^{k}x^{n}a_{n} $, for all $x\in \mathbb B^4(0,1)$, then $f(x)=\sum_{n=0}^{k}\xi_n(x) a_{n} $, for all $x\in \mathbb B^4(0,1)$, belongs to $\mathcal M ( \mathbb B^4(0,1))$. 
\end{enumerate}
\end{ex}

We conclude this subsection with some properties that the $\nu$ operator offers to hyperholomorphic functions.

\begin{cor}\label{identitiesH}Let $\Omega \subset \mathbb H$ be a domain and $f\in \mathcal M (\Omega)$. Then
		\begin{itemize}
			\item [1.] $				-2  \nu [f]=  \mathcal D[ { \bf x} f] + 3 f$,		
				\item  [2.] $			\nu [f]=  \sum_{k=0}^3[ {\bf x}, \textrm{Grad} f_k  ]{\bf e}_k$.
		\end{itemize} 
\end{cor}
\begin{proof} Follows from Proposition \ref{identities}.
\end{proof}

\begin{cor}\label{formulatipoBorelH}Let $\Omega\subset \mathbb H$ be a domain such that 
 $\partial \Omega$ is a 3-dimensional smooth surface and $f\in \mathcal M( \Omega) \cap C(\overline{\Omega},\mathbb H)$. If ${ \bf x}\mathcal D[f]   + 3 f\in L^{2}(\Omega,\mathbb{H})$,   then   
\begin{align*}  &  \int_{\partial \Omega}K (y-x)\sigma_{y}   \bigg( { \bf y} f (y) + 3\mathcal T [    f ]  (y) \bigg)     +2 
\int_{\Omega}  K  (y-x)   \nu [f] (y)  dy  \\
& - 3 \mathcal T\big[ f\big] (x) =  { \bf x} f(x) ,  
\end{align*} 
if  $x\in \Omega$. On the other hand, 
\begin{align*}  &  \int_{\partial \Omega} K (y-x)\sigma_{y}   \bigg( { \bf y} f (y) +3 \mathcal T[ f]  (y) \bigg)    =- 2 
\int_{\Omega}  K  (y-x)   \nu [f] (y)  dy   ,
\end{align*} 
if  $ x\in \mathbb H\setminus\overline{\Omega}$. In addition,  
\begin{align*} \int_{\partial \Omega}  \sigma _x \bigg( 
{ \bf x} f (x) +3 \mathcal T\big[ f\big]  (x)
\bigg)= 
 &  -2 \int_{\Omega } 
  \nu [f] dx.
\end{align*}
 \end{cor}
\begin{proof}Follows from Proposition \ref{formulatipoBorel}. 
\end{proof}

 \begin{cor}\label{identitiesLaplaceoneH} Let $\Omega\subset \mathbb H$ be a domain.
 	 If $f\in \mathcal M(\Omega)$,  then
\begin{itemize}
	\item  [1.] $
	\Delta_{\mathbb R^4}[ { \bf x} f] =	-2 \overline{\mathcal D}\bigg[ \nu [f]\bigg]- 3 \overline{\mathcal D} f$, 
	\item [2.] The function 
	$\dfrac{{ \bf x}}{||{ \bf x}||^{2}}\bigg(\nu [f] -\sum_{k=0}^3\big[ {\bf x}, \textrm{Grad} f_k  \big]{\bf e}_k\bigg)$ is anti-hyperholomorphic on   $\{x\in  \Omega  \ \mid \ {\bf x}\neq 0\}$, i.e.,
$$
		 \overline{\mathcal D}\bigg[\dfrac{{ \bf x}}{||{ \bf x}||^{2}}\bigg(\nu [f] -\sum_{k=0}^3[ {\bf x}, \textrm{Grad} f_k  ]{\bf e}_k\bigg)\bigg]=0,$$
				on $\{x\in  \Omega  \ \mid \ {\bf x}\neq 0\}$.
\end{itemize} 
\end{cor}
\begin{proof} Use Proposition \ref{identitiesLaplaceone}
\end{proof}

\subsection*{Conclusions and future works}

We found some relationships between operators $G,\, \mathcal{D}, \, \nu, \,\Delta_{\mathbb R^4}$ and $\mathcal{T}$ gave us  information about $\textrm{Ker}(G)$ and $\textrm{Ker}( \mathcal{D})$  such as the discovery of a relationship between harmonicity with the slice regular functions and a new subfamily of hyperholomorphic functions associated with $B^4(0,1)$.

An important future work will be to describe the function sets $Ker(G)\cap Ker (\mathcal D) \cap C^1(\Omega,\mathbb H)$  for domains $\Omega\subset \mathbb H$  different {from} $\mathbb B^4(0,1)$ to complement Proposition \ref{GDkerConstant}.
        
This work opens up new avenues for further research on properties of functions presented in Corollary \ref{serieg} and Remark \ref{remsubfam}. On the other hand, new type of series expansion for slice regular functions given in \cite{Stoppato} can be used to find different series expansions of the hyperholomorphic functions than presented in Proposition \ref{serieg} and Remark \ref{remsubfam}.

Future research could explore the extension to higher dimensions context using Clifford algebras and proving analogous results for monogenic function theory and the slice monogenic function theory through some relationships between the global monogenic operator and the Dirac operator. 

\subsection*{Conflict of interest} No conflicts of interest are disclosed by the authors.
\subsection*{Funding} Authors would like to thank Secretar\'ia de Investigaci\'on y Posgrado. Instituto Polit\'ecnico Nacional (grant numbers IND-2026--0101, SIP20260491) for financial assistance.
\subsection*{Author contributions} J.O.G.C. conceived of the presented idea. Both J.O.G.C. and D.G.C. performed the computations. J.B.R. and B.S. supervised the project. All authors discussed the results and contributed to the final manuscript.


\begin{thebibliography}{100}
\bibitem{A1} Alpay, D., Luna-Elizarraras, M.E., Shapiro, M., Struppa, D. \textit{Gleason's problem, rational functions and spaces of left-regular functions: the Split-Quaternioin settings}. Israel J. Math. 226, (2018), 319-349.

\bibitem{ACS} Alpay, D., Colombo, F., Sabadini, I. \textit{Quaternionic Hilbert Spaces and Slice Hyperholomorphic Functions. Operator Theory: Advances and Applications},  Advances and Applications, 304. Birkhäuser/Springer, Cham, [2024], xi+348 pp.

\bibitem{ACS2}  Alpay, D.C.,   Colombo, F.,  Sabadini, I. \textit{Quaternionic de Branges Spaces and Characteristic Operator Function} SpringerBriefs in Mathematics. Springer, Cham, [2020], x+116 pp.

\bibitem{ADS1} Alpay, D., Diki, K., Sabadini, I. \textit{On slice polyanalytic functions of a quaternionic variable}. Results Math. 74, no. 1, Paper No. 17, 25 pp. (2019).

\bibitem{ADS2} Alpay, D., Diki, K., Sabadini, I. \textit{On the global operator and Fueter mapping theorem for slice polyanalytic functions}. Anal. Appl. (Singap.) 19 (6), 941-964 (2021). 

\bibitem{B} {Begehr, H. \textit{Iterated integral operators in Clifford analysis}. Z. Anal. Anwendungen 18 (1999), no. 2, 361-377.}

\bibitem{BDS} Brackx, F., Delanghe, R., Sommen, F. \textit{Clifford Analysis}, Research Notes in Mathematics 76, Pitman, London (1982).

\bibitem{CSSSceBook} Colombo F., Sabadini I., Struppa D.C. \textit{Michele Sce's Works in Hypercomplex Analysis. A Translation with Commentaries}, Birkhäuser/Springer, Cham, [2020], vi+122 pp.

\bibitem{cgs} {Colombo, F., Gentili, G., Sabadini, I., Struppa, D. \textit{Extension results for slice regular functions of a quaternionic variable}. Adv. Math. 222 (2009), no. 5, 1793-1808.}

\bibitem{CSSFueter} Colombo, F., Sabadini, I. and Sommen, F. \textit{The Fueter mapping theorem in integral form and the ${\mathcal F}$-functional calculus}. Math. Meth. Appl. Sci., 33, (2010), 2050-2066.

\bibitem{GlobalOp} Colombo, F., Gonz\'alez-Cervantes, J. O., Sabadini, I. \textit{A nonconstant coefficients differential operator associated to slice monogenic functions}. Trans. Amer. Math. Soc., {365}, 303-318 (2013).

\bibitem{CSS3} Colombo F., Sabadini I., Struppa D.C. \textit{Entire Slice Regular Functions}, SpringerBriefs in Mathematics. Springer, Cham, 2016, v+118.

\bibitem{CSS4} Colombo F., Sabadini I., Struppa D.C. \textit{Noncommutative functional calculus theory and applications of slice hyperholomorphic functions}, Progress in Mathematics, 289. Birkhäuser/Springer Basel AG, Basel, 2011. vi+221 pp.

\bibitem{CAS} Colombo F., De Martino A, Sabadini I. \textit{New Taylor and Laurent Series of Axially Harmonic, Fueter Regular and Polyanalytic Functions}. (2026), arXiv:2505.06555v2.

\bibitem{DSS} Delanghe, R., Sommen, F., Soucek, V. \textit{Clifford analysis and spinor valued functions}, Kluwer Academic Publishers, Dordrecht (1992).

\bibitem{AA} De Martino A, Guzm\'an Ad\'an A. \textit{On the harmonic generalized Cauchy-Kovalevskaya extension and its connection with the Fueter-Sce theorem}, to appear in Trans. Amer. Math. Soc., doi: https://doi.org/10.1090/tran/9614.

\bibitem{AKA1} De Martino A, Diki K, Guzm\'an Ad\'an A. \textit{The Fueter-Sce mapping and the Clifford–Appell polynomials}.  Proc. Edinb. Math. Soc. (2) 66 (2023), no. 3, 642-688.

\bibitem{AKA2} De Martino A, Diki K, Guzm\'an Ad\'an A. \textit{On the connection between the Fueter mapping theorem and the Generalized CK-extension}. Results in Mathematics, 78, 55 (2023), 35pp

\bibitem{DZ} Ding, C. Xu, Z. \textit{Teodorescu Transform for Slice Monogenic Functions and Applications}. Acta Math. Sin. (Engl. Ser.) 42 (2026), no. 8, 2101-2124.

\bibitem{F1} Fueter, R. \textit{Die Funktionentheorie der Differentialgleichungen $\Delta u = 0$ und $\Delta \Delta u = 0$ mit vier reellen Variablen}. Commentarii Mathematici Helvetici, 7(1), (1934). 307-330.

\bibitem{F2} Fueter, R. \textit{Über die analytische Darstellung der regulären Funktionen einer Quaternionenvariablen}. Commentarii Mathematici Helvetici, 8(1) (1935), 371-378.

\bibitem{GS} Gal, S.G.,  Sabadini, I. \textit{Quaternionic Approximation. With Application to Slice Regular Functions}. Frontiers in Mathematics. Birkhäuser/Springer, Cham, 2019. x+221 pp.

\bibitem{GCS}	Gentili, G. Colombo, F.,  Sabadini, I. \textit{A Cauchy kernel for slice regular functions}. Ann. Global Anal. Geom. 37 (2010), no. 4, 361-378.
	
\bibitem{GSSReg} Gentili, G.,  Stoppato, C.,  Struppa, D.C. \textit{Regular functions of a quaternionic variable}. Springer Monographs in Mathematics. Springer, Heidelberg, 2013. x+185 pp. 

\bibitem{GS1} Gentili, G.,   Struppa, D.C. \textit{A new approach to Cullen-regular functions of a quaternionic variable}. C. R. Acad. Sci. Paris, Ser. I, 342, 741-744 (2006).

\bibitem{GS2} Gentili, G., Struppa, D.C. \textit{A new theory of regular functions of a quaternionic variable}. Adv. Math. 216 279-301 (2007).

\bibitem{GPA} Ghiloni, R., \& Perotti, A. \textit{Slice regular functions on real alternative algebras}. Advances in Mathematics, 226(2), 1662-1691, (2011).

\bibitem{GP_2} {Ghiloni, R., Perotti, A. \textit{Global differential equations for slice regular functions}. Math. Nachr. 287 (2014), no. 5-6, 561--573.}

\bibitem{gpr} Ghiloni, R., Perotti, A., Recupero, V. \textit{Noncommutative Cauchy integral formula}. Complex Anal. Oper. Theory. 11 (2), (2017), 289-306.

\bibitem{G} Gonz\'alez-Cervantes, J. O. \textit{On Cauchy Integral Theorem for Quaternionic Slice Regular Functions}. Complex Anal. Oper. Theory, {13} 6, 2527-2539 (2019).

\bibitem{GG1} Gonz\'alez-Cervantes J. O., Gonz\'alez-Campos, D. \textit{The global Borel-Pompeiu-type formula for quaternionic slice  regular  functions}. Complex Var. Elliptic Equ., 66 (5), 1-10 (2021).

\bibitem{GG2} Gonzalez-Cervantes, J. O., Gonz\'alez-Campos D. \textit{On the conformal mappings and the Global operator $G$}. Adv. Appl. Clifford Algebr. 31 (2021), no. 1, Paper No. 6, 13 pp. 

\bibitem{GSp} Guerlebeck, K., Sproessig, W. \textit{Quaternionic and Clifford Calculus for Physicists and Engineers}, Wiley, Chichester, 1998.

\bibitem{KQTS} Kou, K.I., Qian, T., Sommen, F. \textit{Generalizations of Fueter's Theorem}. Meth. Appl. Anal. 9 No. 2 (2002), pp. 273-290.

\bibitem{KVS} {Kravchenko, V. V., Shapiro, M. V. \textit{Integral representations for spatial models of mathematical physics}. Pitman Research Notes in Mathematics Series, 351. Longman, Harlow, 1996}.

\bibitem{PS} Peña-Peña, D., Sommen, F. \textit{A generalization of Fueter's theorem}. Results Math., 49 (3-4) (2006), 301-311.

\bibitem{P} Peña-Peña, D. \textit{Cauchy-Kowalevski extensions}. \textit{Fueter's theorems and boundary values of special systems in Clifford analysis}, PhD thesis, Ghent University, 2008, available at http://hdl.handle.net/1854/11636.

\bibitem{PS2} Peña-Peña, D., Sommen, F. \textit{Fueter's theorem: The saga continues}. J. Math. Anal. Appl. 365 (2010), no. 1, 29--35.

\bibitem{Q} Qian, T. \textit{Generalization of Fueter's result to $\mathbb R^{m+1}$}. Atti Accad. Naz. Lincei Cl. Sci. Fis. Mat. Natur. Rend. Lincei (9) Mat. Appl. 8 (1997), no. 2, 111--117. 

\bibitem{QS2} Qian, T., Sommen, F. \textit{Deriving harmonic functions in higher dimensional spaces}, Z. Anal. Anwend. 22 (2) (2003), 275-288.

\bibitem{Sc} Sce, M. \textit{Osservaziono sulle serie di potenzi nei moduli quadratici}, Atti. Acad. Nax. Lincei Rend Sci Fis Mat. Nat, 23 (1957), 220-225.

\bibitem{shapiro1} Shapiro, M., Vasilevski, N. L. \textit{Quaternionic $\psi$-monogenic functions, singular operators and boundary value problems. I. $\psi-$hyperholomorphy function theory.}  Compl. Var. Theory Appl. {27} (1995), no. 1, 17-46.

\bibitem{shapiro2} Shapiro, M. V., Vasilevski, N. L. \textit{Quaternionic $\psi-$hyperholomorphic functions, singular integral operators and boundary value problems. II. Algebras of singular integral operators and Riemann type boundary value problems}. Complex Variables Theory Appl. {27}, (1995), no. 1, 67-96.

\bibitem{S}  Sommen, F. \textit{On a generalization of Fueter's theorem}, Z. Anal. Anwend. 19 (4) (2000), 899-902.
 
\bibitem{Stoppato} Stoppato, C. \textit{A new series expansion for slice regular functions}. Adv. Math. 231 (2012), no. 3-4, 1401--1416.

\bibitem{sudbery} Sudbery, A. \textit{Quaternionic analysis}. Math. Proc. Cambridge Philos. Soc. 85 (1979), no. 2, 199--224. 

\bibitem{VSMT} Vasilevski N. L., Shapiro M. V. \textit{On an analogue of monogenity in the sense of Moisil-Theodoresco and some applications in the theory of boundary value problems}. Reports of Englarged Session of Seminars of the I.N. Vekua Institute of App. Math. Tbilisi, {1} (1985). 63-66.
\end{thebibliography}
\end{document}